\documentclass[12pt]{article}

\usepackage{amssymb,amsfonts,amsthm,fullpage}

\newtheorem{theorem}{Theorem}[section]
\newtheorem{lemma}[theorem]{Lemma}

\newtheorem{cor}[theorem]{Corollary}
\newtheorem{conj}[theorem]{Conjecture}

\newtheorem{remark}[theorem]{Remark}

\newcommand{\tr}{\mathrm{Tr}}
\newcommand{\C}{\mathbb{C}}
\newcommand{\R}{\mathbb{R}}

\begin{document}

\title{New spectral bounds on the chromatic number encompassing all eigenvalues of the adjacency matrix}
\author{Pawel Wocjan\thanks{Mathematics Department \& Center for Theoretical Physics, Massachusetts Institute of Technology, Boston, USA; on sabbatical leave from Department of Electrical Engineering and Computer Science, University of Central Florida, Orlando, USA; \texttt{wocjan@eecs.ucf.edu}}\quad\quad Clive Elphick\thanks{\texttt{clive.elphick@gmail.com}}}


\date{September 14, 2012}

\maketitle

\begin{abstract}
The purpose of this article is to improve existing lower bounds on the chromatic number $\chi$.  Let $\mu_1,\ldots,\mu_n$ be the eigenvalues of the adjacency matrix sorted in non-increasing order.  

First, we prove the lower bound $\chi \ge 1 + \max_m\{\sum_{i=1}^m \mu_i / -\sum_{i=1}^m \mu_{n - i +1}\}$ for $m=1,\ldots,n-1$.  This generalizes the {\sc Hoffman} lower bound which only involves the maximum and minimum eigenvalues, i.e., the case $m=1$.  We provide several examples for which the new bound exceeds the {\sc Hoffman} lower bound.

Second, we conjecture the lower bound $\chi \ge 1 + S^+ / S^-$, where $S^+$ and $S^-$ are the sums of the squares of positive and negative eigenvalues, respectively.  To corroborate this conjecture, we prove the weaker bound $\chi \ge S^+/S^-$.  We show that the conjectured lower bound is tight for several families of graphs.  We also performed various searches for a counter-example, but none was found.   

Our proofs rely on a new technique of converting the adjacency matrix into the zero matrix by conjugating with unitary matrices and use majorization of spectra of self-adjoint matrices.

We also show that the above bounds are actually lower bounds on the normalized orthogonal rank of a graph, which is always less than or equal to the chromatic number.  The normalized orthogonal rank is the minimum dimension making it possible to assign vectors with entries of modulus one to the vertices such that two such vectors are orthogonal if the corresponding vertices are connected.  

All these bounds are also valid when we replace the adjacency matrix $A$ by  $W * A$ where $W$ is an arbitrary self-adjoint matrix and $*$ denotes the Schur product, that is, entrywise product of $W$ and $A$.
\end{abstract}


\section{Introduction}

Spectral graph theory brings together two apparently unrelated branches of mathematics -- linear algebra and graph theory. Its major goal is to investigate the structure of graphs using the spectra of various matrices associated with graphs \cite{Biggs93,BH12,CRS10,GR01}.

This spectral approach to studying the structure of graphs has its limitations. For example, there exist non-isomorphic graphs that are co-spectral, that is, the eigenvalues of the adjacency matrices of the graphs are the same.  This demonstrates that the spectrum of a graph alone can only provide partial information on the structure of graphs.  
Nevertheless spectral graph theory plays an important role in graph theory. Sometimes a spectral approach provides a novel proof of a familiar result, for example as discussed below for Tur\'an's Theorem.  In other cases a spectral approach leads to results for which no non-spectral proof is known, for example the rarity of Moore graphs. Spectral methods are especially powerful when graphs have symmetry properties. For example, a graph is strongly regular if and only if it has precisely three distinct eigenvalues. 

Spectral graph theory has applications in chemistry, network design, coding theory, computer science and, as described below, in quantum information theory. It even helped Larry Page of Google to become a billionaire, when he patented his PageRank algorithm, which uses the Perron-Frobenius eigenvector of the web graph 
\cite{BL06}.

Determining, and even approximating, the chromatic number of a graph is NP-hard \cite{GJ78,Hastad99}, so attention has focussed on upper and lower bounds. Upper bounds involve finding a coloring but lower bounds are more subtle, because they are based on demonstrating that no coloring exists for a given number of colors. The eigenvalues of a graph provide information about the whole graph whereas degrees provide information about individual vertices. As a result the best known lower bounds for the chromatic number are spectral \cite{Hoffman70}, and in this paper we improve these bounds by incorporating all eigenvalues. We also conjecture a relationship between the sign of the eigenvalues and the chromatic number, which if true could lead to further developments in spectral graph theory. 

Some of the ideas and techniques used in the proofs of these new bounds originated in the context of quantum information theory. We briefly explain below that our new technique of converting the adjacency matrix into the zero matrix by conjugating with unitary matrices is an abstraction of a control-theoretic problem studied in quantum information theory \cite{Wocjan03}.

Examples of matrices associated to graphs are the adjacency matrix, the Laplacian, and the signless Laplacian.  In this paper, we focus on the adjacency matrix.  
Given an undirected graph $G$ with vertex set $V=\{1,\ldots,n\}$ and edge set $E$, the adjacency matrix of $G$ is a matrix $A = (a_{k\ell})$ of size $n$ given by
\[
a_{k\ell} = \left\{
\begin{array}{ll} 
1 & \mbox{if } (k,\ell) \in E \\
0 & \mbox{otherwise.}
\end{array}
\right.
\]
The adjacency matrix $A$ is symmetric, and much is known about the spectra of such
matrices. For instance, the eigenvalues of $A$ are real numbers, which we denote by
$\mu_1,\ldots,\mu_n$, sorted in non-increasing order. 


The {\sc Hoffman} lower bound on the chromatic number \cite{Hoffman70}   
\[
\chi \ge 1 + \frac{\mu_1}{-\mu_n}
\]
is one of the best known results in spectral graph theory.  
There are several tests that a new lower bound for the chromatic number should satisfy to be of interest. These are that the bound:
\begin{itemize}
\item is exact for some class(es) of graphs;
\item exceeds the clique number for some graphs; and
\item exceeds the Hoffman lower bound for the chromatic number for some graphs.
\end{itemize}
A different sort of test is how well it performs for random graphs. 

We prove a new lower bound that generalizes the {\sc Hoffman} lower bound and conjecture a new lower bound that satisfy the above tests.


\begin{theorem}\label{thm:genHoffman}
The chromatic number is bounded from below by
\[
\chi \ge 1 + \max_{m=1,\ldots,n-1} \left\{ \frac{\sum_{i=1}^m \mu_i}{-\sum_{i=1}^m \mu_{n-i+1}} \right\}.
\]
\end{theorem}
This bound reduces to the Hoffman bound when restricted to $m=1$.

To formulate our conjectured lower bound on the chromatic number, we need to introduce some further notation. The inertia of $A$ is the ordered triple $(\pi,\nu,\delta)$, where $\pi$, $\nu$ and $\delta$ are the numbers (counting multiplicities) of positive, negative and zero eigenvalues of $A$ respectively.   Let 
\[
S^+=\mu_1^2+\ldots+\mu_\pi^2.
\]
Let $m=|E|$ denote the number of edges.  Since $\sum_{i=1}^n \mu_i^2 = \tr(A^2) = 2m$, it follows that 
\[
S^-=\mu_{n-\nu+1}^2+\ldots +\mu_n^2 = 2m- S^+.
\]


\begin{conj}\label{conj}
The chromatic number is bounded from below by
\[
\chi \ge 1 + \frac{S^+}{S^-}.
\] 
\end{conj}

While we were not able to prove the above conjectured lower bound, we can prove the following weaker bound.


\begin{theorem}\label{thm:weaker}
The chromatic number is bounded from below by 
\[
\chi \ge \frac{S^+}{S^-}.
\]
\end{theorem}

It is possible to further improve the above generalized {\sc Hoffman} bound and the above weaker bound.  Let $W$ be an arbitrary self-adjoint matrix.  Denote by $\mu_1(W*A),\ldots,\mu_n(W*A)$ the eigenvalues of the Schur product $W*A$, ordered in non-increasing order.  

Theorem~\ref{thm:genHoffman} and the following more general result were first proved by {\sc Wocjan} in the unpublished PhD thesis \cite{Wocjan03}.


\begin{theorem}\label{thm:genHoffmanW}
The chromatic number is bounded from below by
\[
\chi \ge 1 + \max_W \max_{m=1,\ldots,n-1} \left\{ \frac{\sum_{i=1}^m \mu_i(W*A)}{-\sum_{i=1}^m \mu_{n-i+1}(W*A)} \right\},
\]
where $W$ ranges over all self-adjoint matrices and $\mu_1(W*A),\ldots,\mu_n(W*A)$ are the eigenvalues of the Schur product $W*A$ sorted in non-increasing order.
\end{theorem}

We recover Theorem~\ref{thm:genHoffman} by restricting to the case where $W$ is the all-one-matrix $J$.  
We also point out that Theorem~\ref{thm:genHoffmanW} can be seen as a generalization of {\sc Lovasz}'s result \cite{Lovasz79} in which $W$ ranges over all symmetric matrices and $m=1$.

The following lower bound on the chromatic number due to {\sc Barnes} \cite{Barnes01} can also be understood as a special case of Theorem~\ref{thm:genHoffmanW}.


\begin{theorem}\label{thm:barnes}
The chromatic number is bounded from below by
 \begin{equation}\label{eq:barnes}
 \chi \ge 1 + \max_D \{ \mu_1(D^{-\frac{1}{2}} A D^{-\frac{1}{2}}) \},
 \end{equation}
where $\mu_1(D^{-\frac{1}{2}} A D^{-\frac{1}{2}})$  denotes the maximum eigenvalue of  $D^{-\frac{1}{2}} A D^{-\frac{1}{2}}$ and $D=\mathrm{diag}(d_1,\ldots,d_n)$ ranges of all diagonal matrices such that $A+D$ is positive semidefinite. (The condition $A+D$ is positive semidefinite implies that all $d_1,\ldots,d_n$ are necessarily positive so that we can form the matrix $D^{-\frac{1}{2}}$).

Moreover, the maximization process over such $D$ can be solved efficiently with linear programming.
\end{theorem}

The {\sc Hoffman} bound occurs as a special case of this theorem by setting $D=-\mu_n I$.   
The lower bound in eq.~(\ref{eq:barnes}) follows as a special case of Theorem~\ref{thm:genHoffmanW}, which is seen by setting $W=(w_{ij})=(d_i^{-\frac{1}{2}} d_j^{-\frac{1}{2}})$.  An important aspect of this result is that when $W$ is restricted to have this special form and $m=1$, then the maximization process over such restricted $W$ can be solved efficiently with the help of linear programming.  It is not clear how to obtain the best possible $W$ and $m$ when no restrictions are placed on $W$ and $m$.


\begin{theorem}\label{thm:weakerW}
The chromatic number is bounded from below by 
\[
\chi \ge \max_W \left\{ \frac{S^+(W*A)}{S^-(W*A)} \right\},
\]
where $W$ ranges over all self-adjoint matrices, $S^+(W*A)$ is the sum of the squares of the positive eigenvalues $\mu_1(W*A),\ldots,\mu_\pi(W*A)$ and $S^-(W*A)$ the sum of the squares of the negative eigenvalues $\mu_{n-\nu+1}(W*A),\ldots,\mu_n(W*A)$ of the Schur product $W*A$.
\end{theorem}

The paper is organized as follows. In Section~\ref{sec:genHoffman}, we prove the generalized Hoffman bound (Theorems \ref{thm:genHoffman} and \ref{thm:genHoffmanW}).  First, we present a new technique of converting the adjacency matrix into the zero matrix by conjugating with certain diagonal unitary matrices that are constructed from Fourier matrices. Second, we review some basic results of majorization theory of spectra of self-adjoint matrices.  Third, we apply these majorization results to a matrix equality obtained by the conversion technique to derive the generalized Hoffman bound.
In Section~\ref{sec:weaker}, we prove Theorem~\ref{thm:weakerW} by finding a suitable upper bound on the trace of terms occurring in a matrix equality obtained by the conversion technique.  
In Section~\ref{sec:evidence}, we present evidence for Conjecture~\ref{conj}.  
In Section~\ref{sec:performance}, we discuss the performance of the generalized Hoffman bound and Conjecture~\ref{conj}.  


\section{Generalized {\sc Hoffman} bound}\label{sec:genHoffman}

\subsection{Conversion of the adjacency matrix $A$ into the zero matrix}
Recall that $G$ is colorable with $c$ colors if there exists a map $\Phi : V \rightarrow C=\{1,\ldots,c\}$ such that  $a_{k\ell}=1$ implies $\Phi(k)\neq\Phi(\ell)$ for all $k,\ell\in V$.  In words, the graph can be colored with $c$ colors if it is possible to assign at most $c$ different colors to its vertices such that any two adjacent vertices receive different colors.  The chromatic number $q$ is the minimum number of colors required to color the graph.

The following new technique of converting the adjacency matrix into the zero matrix is at the heart of our new bounds.  Here $\dagger$ denotes the operation of taking the transpose of a matrix and changing its entries to their complex conjugates.


\begin{theorem}\label{thm:conversion}
Assume that there exists a coloring of $G$ with $c$ colors.  Then, there exist $c$ diagonal unitary matrices $U_1, \ldots, U_c$ whose entries are $c$th roots of unity such that
\begin{equation}
\sum_{s=1}^{c}  U_s (W * A) U_s^{\dagger}  = 0,
\end{equation}
where $W$ is an arbitrary self-adjoint matrix and $W*A$ denotes the entry-wise product of $W$ and $A$.  

Moreover, this equality remains valid if we replace these diagonal unitary matrices $U_s$ by their inverses $U_s^\dagger$ for $s=1,\ldots,c$.
\end{theorem}

We also formulate the following simple corollary since it leads to the proof of Theorem~\ref{thm:genHoffmanW}  and Theorem~\ref{thm:weakerW}.


\begin{cor}\label{cor:conversion}
Assume that there exists a coloring of $G$ with $c$ colors.  Then, there exist $c-1$ diagonal unitary matrices $U_1,\ldots, U_{c-1}$ whose entries are $c$th roots of unity such that
\begin{equation}\label{cor:reversal}
\sum_{s=1}^{c-1} U_s (-W * A) U_s^\dagger  = W*A.
\end{equation}
Moreover, this equality remains valid if we replace these diagonal unitary matrices $U_s$ by their inverses $U_s^\dagger$ for $s=1,\ldots,c-1$.
\end{cor}

%
%

\begin{remark}
We point out that the ``sign reversal map'' $-W*A\mapsto W*A$ described in the corollary above is an abstraction
of a problem in quantum control theory \cite{JK99,Wocjan03}. A closed quantum system evolves according to $\psi(t)=e^{-iHt}\psi(0)$, where the self-adjoint operator $H$ is the system Hamiltonian and $\psi(t)$ is a vector specifying the configuration of the quantum system at time $t$.

The quantum system can be made to evolve backwards in time by interspersing its natural time evolution $e^{-iHt}$ with certain external control operations which correspond to unitary operations.  This task amounts to effectively changing the Hamiltonian $H$ to $-H$.   

The $k$th vertex of the graph corresponds to the $k$th subsystem and the presence of an edge between vertices $k$ and $\ell$ indicates that $H$ couples the corresponding  subsystems. The unitary matrix $U_s$ corresponds to an external control operations and $s$ indicates the step of the overall control sequence.

It can be shown that the cost of inverting the time evolution is bounded from above by a quantity related to the chromatic number of the 
graph characterizing the coupling topology of $H$ and is bounded from below by an expression involving eigenvalues of $H$.  These results correspond to the generalized {\sc Hoffman's} lower bound on the chromatic number.
\end{remark}

Before proving the theorem and corollary we need to define Fourier matrices.  For each $c$, the matrix
\[
F_c = \left(
\begin{array}{cccc}
\zeta^{1 \cdot 1}      & \zeta^{1\cdot 2}       & \ldots  & \zeta^{1\cdot c} \\
\zeta^{2 \cdot 1}      & \zeta^{2\cdot 2}       & \ldots  & \zeta^{2\cdot c} \\
\vdots                      & \vdots                     &  \ddots &  \vdots                        \\
\zeta^{c \cdot 1}      & \zeta^{c\cdot 2}       & \ldots  & \zeta^{c\cdot c} \\
\end{array}
\right),
\]
where $\zeta$ is an arbitrary primitive $c$th root of unity, is called a Fourier matrix.\footnote{In the literature, it is often customary to start the row and column indices at $0$ instead of $1$.  For our purposes, it is more convenient to choose $1$.}  We have $F_c F_c^\dagger = cI$, that is, its rows are orthogonal vectors and $F_c^\dagger F = c I$, that is, its columns are orthogonal vectors.  Observe that the last row of $F_c$ is equal to the all-one row vector and the last column of $F_c$ is equal to the all-one column vector. 

%
%

\begin{proof} (Theorem~\ref{thm:conversion})
Let $\Phi : V=\{1,\ldots,n\} \rightarrow C=\{1,\ldots,c\}$ be a coloring of $G$.  The entries of the $s$th column of $F_c$ determine the entries of diagonal unitary matrix  $U_s$.  More precisely, the $k$th diagonal entry of $U_s$, which corresponds to vertex $k$,  is set to be $\Phi(k)$th entry of the $s$th column of $F_c$.
Therefore, vertices that have the same color select the same row of $F_c$ and vertices that have different colors select different rows. 

For $s=1,\ldots,c$, we set 
\[
	U_s = \mathrm{diag}(\zeta^{\Phi(1)\cdot s},\zeta^{\Phi(2)\cdot s},\ldots,\zeta^{\Phi(n)\cdot s}) .
\]
Observe that the multiplication of $W*A$ by the diagonal unitary matrix $U_s$ from the left corresponds to the multiplication of the $k$th row of $W*A$ by the $k$th diagonal entry of $U_s$ for $k=1,\ldots,n$.  Similarly, the multiplication of $W*A$ by the diagonal unitary matrix $U_s^\dagger$ from the right corresponds to the multiplication of the $\ell$th column of $W*A$ by the $\ell$th diagonal entry of $U_s^\dagger$ for $\ell=1,\ldots,n$.

Therefore, the $(k,\ell)$th entry of the sum of the conjugates 
\[
\sum_{s=1}^{c} {U_s} (W*A) U_s^\dagger
\]
is equal to 
\[
\sum_{s=1}^{c} \zeta^{\Phi(k)\cdot s} w_{k\ell} a_{k\ell} \, \zeta^{-\Phi(\ell)\cdot s}
\]
for $k,\ell=1,\ldots,n$.

If $\Phi(k)=\Phi(\ell)$, then necessarily $a_{k\ell}=0$ because vertices that receive the same color under the coloring $\Phi$ cannot be adjacent.  Consequently, the corresponding entry of the above sum is equal to $0$.  

If $\Phi(k)\neq\Phi(\ell)$, then $a_{k\ell}=1$ or $a_{k\ell}=0$.  Even if $a_{k\ell}=1$ the corresponding entry of the sum is also equal to $0$ since
\[
\sum_{s=1}^{c} 
\zeta^{(\Phi(k)-\Phi(\ell))\cdot s} = 0.
\]
This equality means that the $\Phi(k)$th and $\Phi(\ell)$th rows of $F_c$ are orthogonal.

Observe that replacing the diagonal unitary matrices $U_s$ by their inverses $U_s^\dagger$ corresponds to replacing the Fourier matrix $F_c$ by its adjoint $F_c^\dagger$, which is a Fourier matrix with respect to the primitive $c$th root of unity $\bar{\zeta}$.  
\end{proof}

%
%

\begin{proof}(Corollary~\ref{cor:conversion})  This result follows immediately because $U_c=U_c^\dagger=I$ since the last row and the last column of $F_c$ are the all-one row vector and the all-one column vector, respectively.
\end{proof}

%
%

To further illuminate the process in Theorem~\ref{thm:conversion}, we now consider $\sum_{s=1}^c U_s X U_s^\dagger$ for arbitrary matrices $X$ and not just matrices of the special form $W*A$.

\begin{remark}
Let $\Phi : V \rightarrow C=\{1,...c\}$ be a coloring and $U_1,\ldots,U_c$ the corresponding diagonal unitary matrices constructed as in the proof of Theorem~\ref{thm:conversion}.  For $b=1,...,c$, set
\[
P_b = \sum_{k\in\Phi^{-1}(b)} e_k^\dagger e_k,
\]
where $e_1=(1,0,\ldots,0)^T,\ldots,e_n=(0,\ldots,0,1)^T$ are the standard basis vectors of $\C^n$.  In words, $P_b$ is the orthogonal projector onto the subspace spanned by $e_k$ for which the corresponding vertex $k$ receives the color $b$ under the coloring $\Phi$.

For an arbitrary matrix $X$, we have
\[
\sum_{s=1}^c U_s X U_s^\dagger = c \sum_{b=1}^c P_b X P_b.
\]
The operation taking $X$ to $\mathcal{C}(X):=\sum_{b=1}^c P_b X P_b$ is known in the literature as pinching \cite[Problem II.5.5]{bhatia}.
\end{remark}


\subsection{Majorization of spectra of self-adjoint operators}

We recall some basic definitions and results in majorization.  We refer the reader to \cite[Chapters II and III]{bhatia}.
Let $x=(x_1,\ldots,x_n)$ be an element of $\R^n$.  Let $x^\downarrow$ be the vector obtained by rearranging the coordinates of $x$ in the non-increasing order.  Thus, if $x^\downarrow=(x^\downarrow_1,\ldots,x^\downarrow_n)$, then $x^\downarrow_1\ge\ldots \ge x^\downarrow_n$.

Let $x,y\in\R^n$.  We say that $x$ is majorized by $y$, in symbols $x \prec y$, if
\begin{equation}\label{eq:maj}
\sum_{i=1}^m x^\downarrow_i \le \sum_{i=1}^m y^\downarrow_i
\end{equation}
for $m=1,\ldots,n-1$ and
\[
\sum_{i=1}^n x^\downarrow_i = \sum_{i=1}^n y^\downarrow_i.
\]
Let $A$ be a self-adjoint operator acting on $\C^n$ and $\mu(A)$ denote the vector in $\R^n$ whose coordinates are the eigenvalues of $A$ specified in any order.

Let $A,B$ be two arbitrary self-adjoint operators.  Then, \cite[Corollary~III.4.2]{bhatia} shows that
\begin{equation}\label{eq:majAB}
\mu^\downarrow(A+B)\prec \mu^\downarrow(A) + \mu^\downarrow(B).
\end{equation}

%
%

\subsection{Proof of the generalized {\sc Hoffman} bound}

\begin{proof}(Theorem~\ref{thm:genHoffmanW})
Consider Corollary~\ref{cor:conversion} and the matrix equality
\[
\sum_{s=1}^{c-1} U_s (-W*A) U_s^\dagger = W*A.
\]
Using the result in eq.~(\ref{eq:majAB}) and induction on $c$, we see that 
\[
\mu^\downarrow(W*A) \prec \sum_{s=1}^{c-1} \mu^\downarrow(U_s (-W*A) U_s^\dagger).
\]
Since conjugation of $-W*A$ by the unitary matrices $U_s$ leaves the spectrum invariant, we obtain
\[
\mu^\downarrow(W*A) \prec (c-1) \mu^\downarrow(-W*A).
\]
Note that $\mu^\downarrow (-W*A) = (-\mu_n(W*A),\ldots,-\mu_1(W*A))$.  The result now follows by using the condition in eq.~(\ref{eq:maj}) for $m=1,\ldots,n-1$, dividing both sides by $-\sum_{i=1}^m \mu_{n-i+1}(W*A)$, and adding $1$ to both sides.
\end{proof}

\subsection{Relation to orthogonal representations}

We now strengthen the result of Theorem~\ref{thm:conversion}.  To do this, we need to introduce orthogonal representations of graphs as studied in \cite{CMNSW07,HPSWM10}, which occur in the study of the quantum chromatic number.  
A $d$-dimensional orthogonal representation of $G$ is a map $\Psi : V \rightarrow \C^d$, mapping vertices to $d$-dimensional column vectors such that $a_{k\ell}=1$ implies $\Psi(k)^\dagger \Psi(\ell)=0$ for $k,\ell\in V$.
The orthogonal rank of $G$, denoted by $\xi$, is the minimum $d$ such that there exists an orthogonal representation of $G$ in $\C^d$.  Furthermore, let $\xi'$ be the smallest $d$ such that $G$ has an orthogonal representation in the vector space $\C^d$ with the added restriction that the entries of each vector must have modulus one \cite{CMNSW07}.  We refer to these representations as normalized orthogonal representations and to $\xi'$ as the normalized orthogonal rank.

To see that $\xi'\le \chi$ holds, we color the vertices with $\chi$ colors according to the coloring $\Phi$ and assign the $\Phi(k)$th column of $F_\chi$ to vertex $k$ for $k=1,\ldots,n$.  

We now show there always exist $\xi'$ diagonal unitary matrices $U_1,\ldots,U_{\xi'}$ such that $\sum_{s=1}^{\xi'} U_s A U_s^\dagger = 0$, providing a strengthening of Theorem~\ref{thm:conversion}.  Denote the entries of the vector $\Psi(k)$ associated to vertex $k$ by $\Psi(k)_1,\ldots,\Psi(k)_{\xi'}$.  It is readily verified that the diagonal unitary matrices $U_s=\mathrm{diag}(\Psi(1)_s,\Psi(2)_s,\ldots,\Psi(n)_s)$ make it possible to convert $A$ to the zero matrix.

We now show that for almost all Hadamard graphs $\xi' < \chi$, relying on the ideas presented in \cite{Avis06}.  A Hadamard graph $G_N=(V,E)$ is the graph with vertex set $V=\{0,1\}^N$ and edge set $E=\{(u,v)\in V\times V \mid d_H(u,v)=N/2\}$, where $d_H$ is the Hamming distance.  For each $k=(k_1,\ldots,k_N)\in V$, define an orthogonal representation by setting $\Psi(k)=((-1)^{k_1},\ldots,(-1)^{k_N})$.  For any pair of adjacent vertices $k$ and $\ell$, we have $\Psi(k)^\dagger\Psi(\ell)=\sum_{s=1}^N (-1)^{k_s+\ell_s} = 0$.  This shows $\xi'(G_N) \le N$.  
The result \cite[Theorem~7.1]{GN08} states that $\chi(G_N)=N$ if and only if $N=2^m$ with $m\le 3$, implying that $\chi(G_{12})>12\ge \xi'(G_N)$.

Hadamard graphs are the only known graphs whose normalized orthogonal
rank $\xi'$ is strictly less than their chromatic number $\chi$. It would be interesting to find new families of graphs satisfying this strict inequality.

The above discussion shows that the generalized Hoffman bound (Theorems~\ref{thm:genHoffman} and \ref{thm:genHoffmanW}) and the weaker bound (Theorems~\ref{thm:weaker} and \ref{thm:weakerW}) are both lower bounds on the normalized orthogonal rank $\xi'$, which is always less than or equal to the chromatic number $\chi$.  This is similar to Bilu's result that Hoffman's bound is a lower bound on the vector chromatic number  \cite{Bilu06}, which in turn is bounded from above by the Lovasz $\vartheta$ number of the complementary graph \cite[Theorem~8.1]{KMS98}.  We do not know the relationship between the vector chromatic number and the normalized orthogonal rank.  The problem is that the vector chromatic number is defined using vectors with real entries, whereas the orthogonal rank is defined using vectors with complex entries.


\section{The weaker lower bound $\chi\ge S^+/S^-$}\label{sec:weaker}

%
%

\subsection{Upper bound on the trace of certain matrix expressions}

In this section, we prove an upper bound on the trace of a matrix expression involving an arbitrary self-adjoint matrix $A$ and an arbitrary unitary matrix $U$.  The matrix $A$ corresponds to the adjacency matrix and $U$ to any of the unitary matrices $U_s$ or its inverse $U_s^\dagger$.

Let $A$ be an arbitrary self-adjoint matrix and $U$ an arbitrary unitary matrix, both acting on $\C^n$.  Let 
\[
	A = \sum_{i=1}^n \mu_i v_i v_i^\dagger,
\]
be the spectral resolution of $A$, that is, $\mu_1,\ldots,\mu_n$ are the eigenvalues of $A$ and $v_1,\ldots,v_n$ the corresponding row eigenvectors of unit length.      
We order the eigenvalues of $A$ in non-increasing order, that is, $\mu_1 \ge \mu_2 \ge \ldots \ge \mu_n$.  The eigenvectors $v_i$ can always be chosen so that they form an orthonormal basis of $\C^n$. 

Let $\pi$ denote the number of positive eigenvalues of $A$ and set
\[
	S^+ = \sum_{i=1}^\pi \mu_i^2, \quad\mbox{and}\quad
	S^-  = \sum_{i=n-\nu+1}^n \mu_i^2.
\]
We also work with the vector space $\R^n$.  Let $e_1=(1,0,\ldots,0)^T,\ldots,e_n=(0,\ldots,0,1)^T$ denote the standard basis vectors of $\R^n$.  Set
$x=(|\mu_1|,\ldots,|\mu_n|)^T\in\R^n$.  Recall that the inertia of $A$ is the ordered triple $(\pi,\nu,\delta)$, where $\pi$, $\nu$ and $\delta$ are the numbers (counting multiplicities) of positive, negative and zero eigenvalues of $A$ respectively.
Define the projectors $P=\sum_{i=1}^\pi e_i e_i^\dagger$, and $N=\sum_{i=\pi+1}^n e_i e_i^\dagger$.  Observe that $P$ is the projector onto the subspace spanned by $e_i$ with $\mu_i$ positive, $N$ the projector onto the subspace spanned by $e_i$ with $\mu_i$ non-positive, and $P+N=I$, where $I$ is the identity matrix acting on $\R^n$.
Further, note that 
\[
x = \sum_{i=1}^n |\mu_i| \cdot e_i,\quad \| P x \|^2=S^+,\quad \mbox{and} \quad \|N x \|^2=S^-. 
\]  
We define the matrix $C=(c_{ij})$ whose entries are given by 
\[
c_{ij}= v_i^\dagger U v_j v_j^\dagger U^\dagger v_i 
\]
for $i,j=1,\ldots,n$.  This matrix is doubly stochastic since $v_i$ form an orthonormal basis and $U$ is a unitary matrix.  A doubly stochastic matrix that arises in this manner is called unitary-stochastic (see \cite[Exercise~II.1.11]{bhatia}).

\begin{lemma}\label{lem:lowerBoundMatrixExpr}
Let $A$ be an arbitrary self-adjoint matrix and $U$ an arbitrary unitary matrix, both acting on $\C^n$.  Then, we have
\begin{eqnarray*}
	\tr(|A| U (-A) U^\dagger) 
	& = & 
	x^T NCN x -  x^T NCP x + x^T PCN x - x^T PCP x \\
	& \le & 
	S^- - x^T NCP x + x^T PCN x,
\end{eqnarray*}
where
\[
|A| = \sum_{i=1}^n |\mu_i| \cdot v_i v_i^\dagger
\] 
and 
$x$, $N$, $P$, and $C$ are defined as above.
\end{lemma}
\begin{proof}
We have
\begin{eqnarray}
	\tr(|A| U (-A) U^\dagger) 
	& = & 
	\sum_{i=1}^n v_i^\dagger \, |A| U (-A) U^\dagger v_i 
	= 
	\sum_{i=1}^n |\mu_i| v_i^\dagger U (-A) U^\dagger v_i \nonumber \\
	& = &
	\sum_{i=1}^n |\mu_i| \sum_{j=1}^n (-\mu_j) v_i^\dagger U v_j v_j^\dagger  U^\dagger v_i 
	= 
	\sum_{i=1}^n \sum_{j=1}^n |\mu_i| c_{ij} (-\mu_j) \nonumber \\
	& = &
	x^T C (N-P) x 
	= 
	x^T (N+P) C (N-P) x \nonumber \\
	& = &
	x^T NCN x - x^T NCP x + x^T PCN x - x^T PCP x. \label{eq:fourTerms}
\end{eqnarray}
The first equality holds since the column vectors $v_i$ form an orthonormal basis of $\C^n$.  The second is due to $v_i^\dagger \, |A| = |\mu_i| v_i^\dagger$.  The third is obtained by using the spectral resolution of $-A$.  The fourth is due to the definition of the doubly stochastic matrix $C$.

Birkhoff's theorem \cite[Theorem~II.2.3]{bhatia} states that any doubly stochastic matrix can be written as a convex sum of permutation matrices.  Combining this with the triangle inequality, we conclude $\|C\|\le 1$.  The Cauchy-Schwarz inequality implies 
\[
	x^T NCN x \le \| N x\| \, \| C N x\| \le \| N x\|^2 \|C\| \le \| N x \|^2 = S^-.
\]
Finally, we obtain the desired upper bound by omitting the term $-x^T PCP x$, which is non-positive since all entries of $x^T$, $x$, $P$, and $C$
are non-negative.
\end{proof}

%
%

\begin{cor}\label{cor:selfadjoint}
Let $A$ be an arbitrary self-adjoint matrix and $U$ an arbitrary unitary matrix. Then, we have
\[
\frac{1}{2} \tr( |A| U (-A) U^\dagger) + \frac{1}{2} \tr( |A| U^\dagger (-A) U) \le S^-.
\]
\end{cor}

\begin{proof}
Let $C=(c_{ij})$ and $C'=(c'_{ij})$ be the matrices defined by
\[
c_{ij} = v_i^\dagger U v_j v_j^\dagger U^\dagger v_i \quad\mbox{and}\quad
c'_{ij} = v_j^\dagger U^\dagger v_i v_i^\dagger U v_j.
\]
Inspection of the entries $c_{ji}$ and $c'_{ij}$ show that they are equal, implying that $C'=C^T$.  Therefore the matrix $C''=\frac{1}{2}C + \frac{1}{2}C'$ is symmetric and doubly-stochastic.

By arguing as in the theorem above, we obtain
\[
\frac{1}{2} \tr( |A| U (-A) U^\dagger) + \frac{1}{2} \tr( |A| U^\dagger (-A) U)  
\le 
x^T NC'' N x - x^T NC'' P x + x^T PC'' N x.
\]
Since $C''$ is symmetric, the terms $x^T PC'' N x$ and $x^T NC'' P x$ are equal.  This is seen by observing that
\begin{eqnarray*}
x^T NC'' P x & = & (x^T NC '' P x)^T 
 = 
x^T P^T {C''}^T N^T {(x^T)}^T \\
& = &
x^T PCN x.
\end{eqnarray*}
Therefore, the corresponding two terms cancel each other out and we are left with $S^-$ as upper bound. 
\end{proof}

%
%

\begin{proof}(Theorem~\ref{thm:weakerW})
Corollary~\ref{cor:conversion} implies that there exist $c-1$ diagonal unitary matrices $U_s$ such that
\begin{equation}\label{eq:conversion}
\frac{1}{2} \sum_{s=1}^{c-1}  U_s (-W*A) U_s^\dagger + U_s^\dagger (-W*A) U_s  = W*A.
\end{equation}

To abbreviate notation, we set $B=W*A$, $S^+=S^+(W*A)$, and $S^-=S^-(W*A)$. 
We multiply eq.~(\ref{eq:conversion}) by $|B|$ from the left and take the trace of both sides.  We obtain the following sequence of equalities:
\begin{eqnarray*}
\tr(|B| B)
& = &
\frac{1}{2} \tr \Big( |B|  \sum_{s=1}^{c-1} U_s (-B) U_s^\dagger + U_s^\dagger (-B) U_s \Big) \\
S^+ - S^-  
& = & 
\frac{1}{2} \sum_{s=1}^{c-1} \tr \Big( |B|  U_s (-B) U_s^\dagger \Big) + \tr\Big( |B| U_s^\dagger (-B) U_s \Big)  \\
S^+ 
& = & 
S^- +  \frac{1}{2} \sum_{s=1}^{c-1} \tr \Big( |B|  U_s (-B) U_s^\dagger \Big) + \tr\Big( |B| U_s^\dagger (-B) U_s \Big).
\end{eqnarray*}
Now we can use Corollary~\ref{cor:selfadjoint} to obtain the upper bound
\[
\frac{1}{2} \tr \Big( |B|  U_s (-B) U_s^\dagger \Big) + \frac{1}{2} \tr\Big( |B| U_s^\dagger (-B) U_s \Big) \le S^-
\]
for all $s=1,\ldots,c-1$.  Both arguments lead to the desired upper bound $S^+ \le c S^-$.
\end{proof}

%
%

\begin{remark}
We now briefly explain why it is not possible to prove the conjectured lower bound for all graphs using the current approach.  The natural idea to prove a bound stronger than the weaker bound is to keep the $-x^T P C P x$ terms, which are dropped in Lemma~\ref{lem:lowerBoundMatrixExpr} and Corollary~\ref{cor:selfadjoint}.

Let $C_s$ be the doubly stochastic matrices that arise when we conjugate by $U_s$ for $s=1,\ldots,c-1$.  Unfortunately, it turns out that for the special case $K_n$, the corresponding term $\sum_{s=1}^{c-1} x^T P C_s P x$ is equal to zero.  Therefore the additional ``undesired'' $S^-$ term cannot be offset by $-\sum_{s=1}^{c-1} x^T P C_s P x$.
\end{remark}

\section{Evidence for the conjectured lower bound $\chi\ge 1 + S^+/S^-$}\label{sec:evidence}

We start by describing how the conjectured lower bound is related to known results in spectral graph theory.  We describe a hierarchy of lower bounds starting from the weakest and ending at the conjectured lower bound.

{\sc Myers} and {\sc Liu} \cite{ML:72} proved the following degree-based bound
\[
\omega \ge 1 + \frac{2m}{n^2-2m}.
\]
on the clique number $\omega$. In 1972, {\sc Cvetkovic} \cite{Cvetkovic72} proved that
\[
\chi \ge 1 + \frac{\mu_1}{n-\mu_1}.
\]
{\sc Wilf} \cite{Wilf86} proved that this bound is in fact a lower bound for the clique number and therefore implies the concise Tur\'an theorem. 
In 1983, {\sc Edwards} and {\sc Elphick} proved that
\[
\chi \ge 1 + \frac{\mu_1^2}{2m-\mu_1^2} 
\]
and conjectured that $\chi$ can be replaced by the clique number $\omega$ \cite{EE83,Elphick81}.  {\sc Nikiforov} proved this conjecture in \cite{Nikiforov02} and generalized it by replacing $\mu_1^2$ by $\mu_1^r$ and $2m$ with the number of $r$-walks in $G$ \cite{Nikiforov06}. 

{\sc Bollobas} and {\sc Nikiforov} conjectured that
\[
\omega \ge 1 + \frac{\mu_1^2 + \mu_2^2}{2m - \mu_1^2 - \mu_2^2}
\]
for non-complete graphs \cite{BN07}.  This conjecture is exact for complete bipartite and complete regular $q$-partite graphs, since $\mu_2=0$ for these graphs.

Note that it is not possible to replace the chromatic number with the clique number in Conjecture~\ref{conj} because, for example, the Coxeter graph provides a counter-example. {\sc Smith} \cite{Smith69} has proved that $\mu_2>0$ for all connected graphs other than complete multipartite graphs, so Conjecture~\ref{conj} is an improvement on the result due to {\sc Edwards} and {\sc Elphick} for all such graphs.   

We have not been able to prove Conjecture~\ref{conj}. In addition to proving the weaker bound in the previous section, we are able to identify four additional types of evidence in support of the conjecture.  These are as follows:
 
First, we prove that $n - \alpha + 1 \ge 1 + S^+/S^-$, where $\alpha$ is the independence number of a graph, and it is well known that 
$n - \alpha + 1\ge \chi$.
\begin{theorem}
We have
\[
n - \alpha \ge
\frac{S^+}{S^-}.
\]
\end{theorem}
\begin{proof}  
Cvetkovic noted that $\pi \le n - \alpha$ and that $\nu \le n - \alpha$ \cite{GR01} Let $T=\mu_1+\ldots+\mu_\pi=-\mu_{n-\nu+1}-\ldots-\mu_n$ is $\frac{1}{2}$ times the energy $E=\tr(|A|)$.  Observe that
\[
S^+ = \mu_1^2 + \ldots + \mu_\pi^2 \le \mu_1(\mu_1+\ldots+\mu_\pi) = \mu_1 T,  
\]
with equality when $\mu_2=0$.  Conversely, using the Cauchy-Schwartz inequality we obtain
\begin{eqnarray*}
T^2 
& = & 
(\mu_{n-\nu+1} + \ldots + \mu_n)^2 \le \nu \cdot (\mu_{n-\nu+1}^2 + \ldots +\mu_n^2) \\
& = & 
\nu \cdot S^- \le (n-\alpha) S^-
\end{eqnarray*}
with equality when $A$ has $(n-\alpha)$ eigenvalues equal to $-T/(n-\alpha)$.

Therefore:
\[
\frac{S^+}{S-} \le  \frac{T\mu_1}{T^2/(n-\alpha)} = \frac{(n-\alpha)\mu_1}{T} \le n-\alpha.
\]
\end{proof}

Secondly, the conjecture is exact for several graph families, including all bipartite, complete and complete regular $q$-partite graphs. For example, any bipartite graph has a spectrum which is symmetric about zero. Therefore for all bipartite graphs: $S^+ = S^-$ and consequently $1 + S^+/S^- = 2$ when $\chi = 2$.  The proofs for complete and complete regular $q$-partite graphs are straightforward.

Thirdly, we have proved the conjecture is correct for strongly regular, complete $q$-partite and Regular Two-graphs and for Kneser graphs $KG_{p,k}$ for $k \le 4$.   

Finally we have conducted various searches for a counter-example. The Wolfram Mathematica 8.0 function \texttt{GraphData[n]} lists named graphs on $n$ vertices. We have searched the thousands of such graphs with $n\le 50$ and found no counter-examples.  Wilf proved the well known upper bound that $\chi \le 1 + \mu_1$. Therefore if there exists a graph for which $S^+/S^- > \mu_1$, this would provide a counter-example. {\sc Godsil} has tested this inequality against all $274,668$ graphs on $9$ vertices using Sage and found no counter-examples \cite{Godsil12}. The conjecture performs particularly well for small, dense random graphs of the form $G_{n,p}$, where $n$ is the number of vertices and $p$ is the independent probability of each edge being present. We have therefore used the Wolfram function \texttt{RandomGraph[n,p]} to generate over $100$ graphs with $n = 10$ and $p = 0.85$ or $0.9$ and again found no counter-examples.


\section{Empirical performance of the bounds}\label{sec:performance}

The best known lower bound for $\chi$ is the {\sc Hoffman} bound. Unlike the {\sc Myers} and {\sc Liu}, {\sc Cvetkovic}, and {\sc Edwards} and {\sc Elphick} bounds, the {\sc Hoffman} bound is not a lower bound for the clique number $\omega$. We have therefore compared the performance of Conjecture~\ref{conj} and Theorem~\ref{thm:genHoffman} with the performance of the {\sc Hoffman} bound. We have focused on the performance of Conjecture~\ref{conj} rather than of Theorem~\ref{thm:weaker} ($\chi\ge S^+/S^-)$ because we regard Theorem~\ref{thm:weaker} as interesting rather than useful. Theorem~\ref{thm:weaker} is occasionally better than the {\sc Hoffman} bound, but we have not been able to find a graph for which Theorem~\ref{thm:weaker} exceeds the clique number. We hope that Theorem~\ref{thm:weaker} will become a stepping stone to a proof of Conjecture~\ref{conj}, rather than a significant result in its own right. Theorem~\ref{thm:genHoffman} is a generalization of the {\sc Hoffman} bound and so can never be worse than the {\sc Hoffman} bound. 
We have not performed the maximization over $W$ in Theorem~\ref{thm:genHoffmanW} and Theorem~\ref{thm:weakerW}. 

We have made comparisons using both named and random graphs in Wolfram Mathematica 8.0. The results are set out below.  


\subsection{Named graphs}    

The Wolfram function \texttt{GraphData[n]} generates parameters for named graphs on $n$ vertices. For example, there are $78$ named graphs on 16 vertices, excluding the complete, empty and bipartite graphs. Tabulated below are the numbers of such named graphs on $16$, $25$ and $28$ vertices and the percentages of these graphs for which Theorem~\ref{thm:genHoffman} and Conjecture~\ref{conj} exceed the {\sc Hoffman} bound:  

\[
\begin{array}{ccccc}
	n		& \mbox{named graphs} & \mbox{Theorem~\ref{thm:genHoffman}} & \mbox{Conjecture~\ref{conj}} \\
	16		&	78				& 15\%			 & 22\% \\
	25		&	30				& 7\%			 & 13\% \\
	28		&	27				& 15\%			 & 19\% 
\end{array}
\]
An example of a graph for which the new bounds perform well is Barbell(8), for which the Hoffman bound is $4.8$, Theorem~\ref{thm:genHoffman} is $5.9$, Conjecture~\ref{conj} is $7.3$ and the chromatic number is $8$.

Theorem~\ref{thm:genHoffman} and Conjecture~\ref{conj} tend to perform particularly well for graphs that are nearly disconnected.

 
\subsection{Random Graphs}

The Wolfram function \texttt{RandomGraph[n,p]} generates a random graph $G_{n,p}$ on $n$ vertices with each edge being present with independent probability $p$. Eigenvalues are found using the function \texttt{Spectrum}, provided the Wolfram package ``Combinatorica'' has been loaded. Theorem~\ref{thm:genHoffman} almost never exceeds the {\sc Hoffman} bound for random graphs, because for almost all random graphs $\mu_1 \gg \mu_2$, and consequently generalizing over more eigenvalues than {\sc Hoffman} does not improve the bound. 
 
Tabulated below is the performance of Conjecture~\ref{conj} against the {\sc Hoffman} bound for each combination of $n = 20$ and $50$ and $p = 0.5$, $0.7$ and $0.9$, in each case averaged over $15$ graphs. We have included a comparison with the 1988 result due to {\sc Bollobas} \cite{Bollobas88} that the chromatic number of almost every random graph $G_{n,p}$ is: $q = (1/2 +o(1))n/\log_b(n)$, where $b = 1/(1-p)$.
 \[
 \begin{array}{ccccc} 
n		& p		& \mbox{Hoffman Bound}	& \mbox{Conjecture~\ref{conj}} & \mbox{Bollobas} \\
20 		& 0.5	&		3.3				& 2.9		& 2.3 \\
20		& 0.7	&		4.3				& 4.2		& 4.0 \\
20 		& 0.9	&		6.3				& 8.2		& 7.7 \\
50		& 0.5	&		4.5				& 3.2		& 4.4 \\
50 		& 0.7	&		6.2				& 4.9		& 7.7 \\
50		& 0.9	&		9.9 				& 10.8		& 14.7 
\end{array}	
 \]
It can be seen that for $n = 20$ both bounds exceed the {\sc Bollobas} formula for varying levels of $p$, because for low levels of $n$ the $o(1)$ term is material.   
The main conclusion is that the performance of Conjecture~\ref{conj} is strongly affected by graph density. For sparse graphs with $p = 0.5$ the {\sc Hoffman} bound is almost always better, irrespective of $n$. For dense graphs the position is more complex. With $p = 0.9$,  Conjecture~\ref{conj} usually exceeds {\sc Hoffman} for $n$ less than about $65$ but is worse than {\sc Hoffman} for $n$ greater than about $65$. 
The range of values of both bounds is fairly small. With a sample size of 15, for n = $50$ and $p = 0.9$ the range of the {\sc Hoffman} bound was $9.1$--$11.3$ and the range for Conjecture~\ref{conj} was $9.6$--$11.5$.      	


\section{Conclusions}

Most spectral bounds in graph theory involve a small number of eigenvalues. In this paper we have investigated two new lower bounds for the chromatic number, which involve all eigenvalues of the adjacency matrix. The principal open question raised by the paper is whether Conjecture~\ref{conj} is true. If it is then it provides an unexpected relationship between the sign of the eigenvalues of the adjacency matrix of a graph and its chromatic number. This relationship does not apply to the clique number.

Underpinning our new bounds is Theorem~\ref{thm:conversion}, which is a new characterization of a $\chi$-chromatic graph. 

{\sc Bilu} has proved that the {\sc Hoffman} bound is also a lower bound for the vector chromatic number \cite{Bilu06}. We have shown that the generalized Hoffman bound is also a lower bound on the normalized orthogonal rank.  The exact relationship between the vector chromatic number and the normalized orthogonal rank is an open question.

Finally we have not considered how to efficiently maximize over $W$ in Theorems~\ref{thm:genHoffmanW} and \ref{thm:weakerW}, but the paper by {\sc Barnes} \cite{Barnes01} provides an indication on how to proceed.


\section*{Acknowledgements}
We would like to thank Dominik Janzing, Giannicola Scarpa and Simone Severini for helpful discussions.
P.W. gratefully acknowledges the support from the National Science Foundation CAREER Award CCF-0746600.  This work was supported in part by the National Science Foundation
Science and Technology Center for Science of Information, under
grant CCF-0939370.

\bibliographystyle{amsplain}

\begin{thebibliography}{10}
\bibitem{Avis06}
D.~Avis, J.~Hasegawa, Y.~Kikuchi and Y.~Sasaki, \emph{A quantum protocol to win the graph coloring game on all Hadamard graphs}, 
Journal IEICE Transactions on Fundamentals of Electronics, Communications and Computer Sciences, Volume E89--A Issue 5,
pp.~1378-1381, 2006. 

\bibitem{Barnes01}
E.~R.~Barnes, \emph{A lower bound for the chromatic number of a graph}, 
Contemporary Mathematics, 275 (2001), 3--12.

\bibitem{bhatia}
R.~Bhatia, \emph{Matrix analysis}, Graduate text in mathematics, vol. 169, Springer

\bibitem{Biggs93}
N.~Biggs, \emph{Algebraic Graph Theory}, Cambridge University Press, 1993

\bibitem{Bilu06}
Y.~Bilu, \emph{Tales of Hoffman: Three extensions of Hoffman's bound on the graph chromatic number}, J. Combin, Theory, Ser. B, 96 (2006), 608--613.

\bibitem{Bollobas88}
B.~Bollobas, \emph{The chromatic number of random graphs}, Combinatorica 8(1), 49-55, 1988.

\bibitem{BN07}
B.~Bollobas and V.~Nikiforov, \emph{Cliques and the spectral radius}, J. Combin. Theory Ser. B 97 (2007), 859--865.

\bibitem{BH12}
A.~Brouwer and W.~Haemers, \emph{Spectra of graphs}, Universitext, Springer, 2012; \texttt{www.win.tue.nl/\~{}aeb/2WF02/spectra.pdf} 

\bibitem{BL06}
K.~Bryan and T.Leise, \emph{The \$25,000,000,000 eigenvector. The linear algebra behind Google},  SIAM Review, 48 (3) 2006, 569--81; \\
\texttt{http://www.rose-hulman.edu/~bryan/googleFinalVersionFixed.pdf}

%
\bibitem{CMNSW07}
P.~J.~Cameron, A.~Montanaro, M.~W.~Newman, S.~Severini, A.~Winter, \emph{On the chromatic number of graph}, The Electronic Journal of Combinatorics, 14, \#R81, 2007.

\bibitem{Cvetkovic72}
D.~Cvetkovic, \emph{Chromatic number and the spectrum of a graph}, Publ. Inst. Math. (Beograd), 14(28) (1972), 25--38.

\bibitem{CRS10}
D,~Cvetkovic, P.~Rowlinson, and S.~Smic, \emph{An introduction to the theory of graph spectra}, London Mathematical Society Student Texts 75, Cambridge University Press, 2010

\bibitem{EE83}
C.~Edwards and C.~Elphick, \emph{Lower bounds for the clique and the chromatic number of a graph}, Discrete Appl. Math. 5 (1983) 51--64. 

\bibitem{Elphick81}
C.~Elphick, \emph{School Timetabling and Graph Colouring}, PhD thesis (unpublished), University of Birmingham, UK, 1981.  

\bibitem{GJ78}
M.~Gary and D.~Johnson, \emph{Computers and Intractability: A guide to the theory of NP-completeness},  Freeman 1978.

\bibitem{GN08}
C.~D.~Godsil and M.~W.~Newman, \emph{Colouring an orthogonality graph}, 
SIAM Journal on Discrete Mathematics 22(2) (2008), pp.--683-692. 

\bibitem{Godsil12}
C.~Godsil, private correspondence, 2012.

\bibitem{GR01}
C.~Godsil and G.~Royle, \emph{Algebraic Graph Theory}, Springer-Verlag, New York, 2001. 

\bibitem{HPSWM10}
G.~Haynes, C.~Park, A.~Schaeffer, J.~Webster, L.~H.~Mitchell, \emph{Orthogonal vector coloring}, The Electronic Journal of Combinatorics, 17, \#R55, 2010; {\small \texttt{http://www.combinatorics.org/ojs/index.php/eljc/article/view/v17i1r55/}}

\bibitem{Hoffman70}
A.~J.~Hoffman, \emph{On eigenvalues and colourings of graphs}, in: Graph Theory and its Applications, Academic Press, New York (1970), pp. 79--91.

\bibitem{Hastad99}
J.~Hastad, \emph{Clique is hard to approximate within $n^{1-\varepsilon}$}, Acta Math., 182:105Ð142, 1999.

\bibitem{ML:72}
B.~R.~Myers and R. Liu, \emph{A lower bound for the chromatic number of a graph}, Networks 1 (1972), 273--277.

\bibitem{JK99}
J.~Jones, E.~Knill, \emph{Efficient refocussing of one spin and two spin interactions for NMR quantum computation}, 
J. Magn. Resonance 141 (1999), pp.~322--325

\bibitem{KMS98}
D.~Karger, R.~Motwani and M.~Sudan, \emph{Approximate graph coloring by semidefinite programming}, Journal of the ACM 45(2), pp.~246--265, 1998.

\bibitem{Lovasz79}
L.~Lovasz, \emph{On the Shannon capacity of a graph}, IEEE Trans. Inf. Th., 25(1) (1979), 1--7.

\bibitem{Nikiforov02}
V.~Nikiforov, \emph{Some inequalities for the largest eigenvalue of a graph}, Combin. Probab. Comput. 11 (2002), 179--189. 

\bibitem{Nikiforov06}
V. Nikiforov, \emph{Walks and the spectral radius of graphs}, Linear Algebra Appl. 418, 257-268, 2006. 

\bibitem{Nikiforov07}
V. Nikiforov, \emph{Chromatic number and spectral radius}, Linear Algebra Appl. 426, 810-814, 2007.

\bibitem{Smith69}
J.~H.~Smith, \emph{Some properties of the spectrum of a graph}, pp. 403--406 in: Combinatorial Structures and their Applications, Proc. Conf. Calgary 1969, Gordon and Breach, New York, 1970.  

\bibitem{Wilf86}
H.~Wilf, \emph{Spectral bounds for the clique and independence numbers of graphs}, J. Combin. Theory Ser. B 40(1986), 113--117. 

\bibitem{Wocjan03}
P.~Wocjan, \emph{Computational power of Hamiltonians in quantum computing}, Dissertation (unpublished), University of Karlsruhe, Germany, 2003; \\
 \texttt{http://www.eecs.ucf.edu/\~{}wocjan/dissertation.pdf}

\end{thebibliography}

\end{document}